\documentclass[11pt]{amsart}
\advance\hoffset by -0.5 truecm
\usepackage{amssymb}
\usepackage{graphicx}
\newtheorem{Theorem}{Theorem}[section]
\newtheorem{Lemma}[Theorem]{Lemma}

\newtheorem{Proposition}[Theorem]{Proposition}

\newtheorem{Definition}[Theorem]{Definition}

\def\V{\mbox{Var}}

\def\R\re
\def\V{\bf V}

\def \re{{\mathbb R}}

\def \0{\lambda_{0}}

\begin{document}
\title[Yamabe constants]{A note on Yamabe constants of products with hyperbolic spaces}

\author[G. Henry]{Guillermo Henry}\thanks{G. Henry is supported
by a postdoctoral fellowship of CONICET}
  \address{Departamento de Matem\'atica, FCEyN, Universidad de Buenos
Aires, Ciudad Universitaria, Pab. I., C1428EHA,
           Buenos Aires, Argentina.}
\email{ghenry@dm.uba.ar}

\author[J. Petean]{Jimmy Petean}\thanks{J. Petean is supported
by grant 106923-F of CONACYT}
 \address{CIMAT  \\
          A.P. 402, 36000 \\
          Guanajuato. Gto. \\
          M\'exico \\
           and Departamento de Matem\'{a}ticas, FCEyN \\
          Universidad de Buenos Aires, Argentina (on leave).}
\email{jimmy@cimat.mx}


\date{}

%\maketitle

\begin{abstract} We study
the ${\bf H}^n$-Yamabe constants of Riemannian products $({\bf H}^n
\times M^m , g_h^n +g)$,
where $(M,g)$ is a compact Riemannian manifold of constant scalar curvature and
$g_h^n$ is the hyperbolic metric on ${\bf H}^n$.
Numerical calculations can be carried out due to the uniqueness of
(positive, finite energy) solutions of the equation
$\Delta u -\lambda u +  u^q =0$ on hyperbolic space ${\bf H}^n$
under appropriate bounds on the parameters $\lambda, q$, as shown by
G. Mancini and K. Sandeep.
We do explicit numerical estimates in the cases
$(n,m)=(2,2),(2,3)$ and $(3,2)$.
\end{abstract}

\maketitle

\section{Introduction}

For a closed Riemannian k-dimensional manifold $(W^k ,g)$
the  Yamabe constant
of its conformal class $[g]$ is defined as

$$ Y(W,[g]) = \inf_{h\in [g]} \frac{\int_W s_h \ dv_h }{Vol(W,h)^{\frac{k-2}{k}}}
  $$

\noindent
where $s_h$ is the scalar curvature, $dv_h$ the volume element
and $Vol(W,h) = \int_W   dv_h $ is the volume of $(W,h)$.

We let $a_k =
\frac{4(k-1)}{k-2}$ and
$p=p_k=\frac{2k}{k-2}$. For $h\in [g]$ we write
$h=f^{p-2} g$ for a function $f:W \rightarrow {\re}_{>0}$
and write the previous expression in terms of $f$ and $g$: we have

$$Y(W,[g]) = \inf_{f\in C_+^{\infty}} Y_g (f) ,$$

\noindent
where
$$Y_{g} (f)=\frac{\int_{W}a_{k}|\nabla f|^2+s_g f^2 \ dv_{g}}{\|f\|_{p_{k}}^2}.$$

We will call $Y_g$ the Yamabe functional. Its critical
points are solutions of the Yamabe equation:

$$-a_k \Delta_g f + s_g f = \mu f^{p-1} ,$$

\noindent
where $\mu$ is a constant ($\mu = Y_g (f) {\| f \|}_p^{2-p}$):
this means that the corresponding metric $h=f^{p-2} g$
has constant scalar curvature. The {\it Yamabe problem},
which consists in finding metrics of constant scalar curvature
in a given conformal class, was solved for closed Riemannian
manifolds by showing that the infimum in the definition of
the Yamabe constant is always achieved.

There are different possible ways to try to generalize these ideas to
non-compact manifolds. The non-compact case has attracted the atention
of many authors (see for instance \cite{Botvinnik, Akutagawa, Kim, Ruiz} )
for the interest in the problem
itself and also because non-compact examples play an
important role when studying the Yamabe invariant (the supremum
of the Yamabe constants over the family of conformal classes of
metrics on a fixed closed manifold).

\vspace{.5cm}

In this article we will study the case
when the manifold is a Riemannian product $({\bf{H}}^n\times M^m,g_h^n+g)$
where $(M^m,g)$  is a   closed Riemannian manifold of
constant scalar curvature and $({\bf{H}}^n,g_h^n)$
is the $n-$dimensional hyperbolic space of
 curvature $-1$. We denote by  ${\bf s} =s_g-n(n-1)$  the scalar curvature of
 $g_h^n+g$. We define their Yamabe constant as:

$$Y({\bf{H}}^n\times M^m,g_h^n+g) =
\inf_{f \in  L_1^2 ({\bf{H}}^n \times M^m) - \{ 0 \}} \ \ Y_{g_h^n+g} \  (f).$$

\noindent
Note that this is well defined since the Sobolev embedding
 $L_1^2 ({\bf{H}}^n \times M^m) \subset L^p ({\bf{H}}^n \times M^m)$
holds (see \cite[Theorem 2.21]{Aubin}).

It is important the case when $(M,g)$ is $(S^m,rg_0^m)$,
where $g_0^m$ is the round metric of constant curvature 1 and
$r$ is a positive constant , since it plays a fundamental
role in understanding the
behaviour of the Yamabe invariant under surgery
(see \cite[Theorem 1.3 and Section 3]{Ammann}). It can be seen via
symmetrizations that the
infimum in the definition is achieved by a function
which is radial in both variables (it depends only on the distance
to the origin in hyperbolic space and on the distance to a fixed
point in the sphere).
It has been conjectured \cite{Akutagawa, Ammann} that
the minimizer actually depends only on the
 ${\bf{H}}^n$-variable. The main objective of this article is
to show that if this were the case then one could compute the
corresponding Yamabe constants numerically.

We recall the following definition from \cite{Akutagawa}:

\begin{Definition} For a Riemannian product $(N\times M, h+g)$ we define the
$N$-Yamabe constant as

$$Y_N (N\times M, h+g) =  \inf_{f\in L_1^2 (N) -\{ 0 \} } Y_{h+g} (f) .$$

\end{Definition}

We will study  $Y_{{\bf{H}}^n}
({\bf{H}}^n\times M^m,g_h^n+g)$, where $(M,g)$ is a closed Riemannian
manifold of
constant scalar curvature $s_g$ and volume $V_g$. Note that

$$Y_{{\bf{H}}^n}
({\bf{H}}^n\times M^m,g_h^n+g) =  V_g^{\frac{2}{m+n}} \ \inf_{f\in L_1^2 ({\bf{H}}^n ) }
\frac{\int_{{\bf{H}}^n}a_{n+m}|\nabla f|^2+ {\bf s} f^2 \
dv_{g_h^n}}{\|f\|_{p_{n+m}}^2}. $$

If $f\in L_1^2 ({\bf{H}}^n )$ is a critical point of $Y_{g_h^n + g}$
restricted
to  $ L_1^2 ({\bf{H}}^n )$, then it satisfies the subcritical equation

$$-a_{n+m} \Delta_h f + {\bf s} f = \mu f^{p_{m+n} -1}
\ \   $$

\noindent
where $\mu $ is a constant  (it is called a subcritical equation
since $p_{m+n} <p_n$).

Let

$$c_{m,n} =(n-1)(m-1)/(m+n-2) .$$

In Section 2 we will prove:

\begin{Theorem} If $s_g >  c_{m,n}$ then $Y_{{\bf{H}}^n}
({\bf{H}}^n\times M^m,g_h^n+g) > 0$ and the constant is achieved.
If $s_g =  c_{m,n}$
 then $Y_{{\bf{H}}^n}
({\bf{H}}^n\times M^m,g_h^n+g) > 0$ but the constant is not achieved.
If $s_g <  c_{m,n}$ then $Y_{{\bf{H}}^n}
({\bf{H}}^n\times M^m,g_h^n+g) =-\infty $.

\end{Theorem}

\vspace{.5cm}

In Section 3 and Section 4 we
will consider the case $(M,g) = (S^m , rg_0^m )$, for $r\in [0,1]$,
$m\geq 2$.
The case when $r\in [0,1]$ is of interest because these are the values
that appear in the surgery formula \cite{Ammann}. Note that for
$r\in (0,1]$ we have  $s_{rg_0^m} =(1/r)m(m-1) >c_{m,n}$.
Let $g(r) = g_h^n + rg_0^m $ and denote by $g^n_e$ the Euclidean metric
on $\re^n$. We define
 $Q_{n,m}:[0,1] \longrightarrow \re_{>0}$ by

$$Q_{n,m}(r)=\left\{ \begin{array}{ll}
 Y_{{\bf{H}}^n}({\bf{H}}^n\times S^m,g(r)) &  if\ r>0, \\
                                     Y_{\re^n}(\re^n\times S^m,g^n_e+g_0^m) & if\ r=0.
                                  \end{array}
                                \right.$$

In section 2 we will also show the following:

\begin{Proposition}
$Q_{n,m}$ is a continuous function.
\end{Proposition}

Note that  $Q_{n,m}(0)$ is computed in
\cite[Theorem 1.4]{Akutagawa}
in terms of the best constants in the Gagliardo-Nirenberg
inequalities (which can be computed numerically).

If $(M^m ,g)$ is a closed Riemannian manifold of
constant scalar curvature $s_g >c_{m,n}$
and  $f$ realizes   $Y_{\bf{H^n}}({\bf{H}}^n\times M^m,g_h^n+g)$,
then $f$ is
a positive smooth solution of the subcritical  Yamabe equation

$$-a_{m+n} \Delta_{g_h^n} f + s_{g^h_n+g}  f = a_{m+n}f^{p_{m+n}-1} $$

\noindent
(of course, $f$ is a minimizer then for any positive constant
$\alpha$, $\alpha f$ is also a minimizer. One obtains a solution
of the previous equation by picking $\alpha$ appropriately).

Due to the symmetries of hyperbolic
space, using symmetrization,
one can see that $f$ is a
radial function (with respect to some fixed point).
Consider the following model for hyperbolic space:

$${\bf H}^n = (\re^n , \sinh^2  (r) \ g_0^{n-1 } + dr^2 ).$$

\noindent
For a radial function $f$ write $f(x)=\varphi (\| x \| )$,
where $ \| x \|$ denotes the distance to the fixed point.
Then $f$ is a solution of the Yamabe equation if
$\varphi : [0, \infty ) \rightarrow \re_{>0}$ solves
the ordinary differential equation:

$$EQ_{\lambda ,n,q}  :  \  \  \  \varphi '' +(n-1) \frac{e^{2t}+1 }{e^{2t} -1}  \varphi ' = \lambda \varphi - \varphi^q $$

\noindent
where $\lambda=s_{g^n_h+g}/a_{n+m}$ and $q=p_{n+m}-1$.

Note that

$$\int_{{\bf H}^n} f^k dv_{g_h^n} =V_{g_0^{n-1}}  \int_0^{\infty} \varphi^k (t) \sinh^{n-1} (t) dt $$

\noindent
(for any $k>0$) and

$$\int_{{\bf H}^n} {\| \nabla f}^2  \| dv_{g_h^n}  =  V_{g_0^{n-1}} \int_0^{\infty} \varphi '^2  \sinh^{n-1} (t) dt .$$

Uniqueness of (positive, finite energy) solutions
of the subcritical Yamabe equation
(or equivalently $EQ_{\lambda ,n,q}$) was proved
by G. Mancini and K. Sandeep in
\cite[Theorem 1.3, Theorem 1.4]{Mancini}.
We will describe the solutions of the ODE in Section 3
to see that one can numerically compute $Q_{n,m} (r)$ for any
fixed $r\in (0,1]$
and use this in Section 4 to prove:

\begin{Theorem} For  $(n,m) =(2,2), (2,3), (3,2) $and $r\in [0,1]$
$Q_{n,m}(r) \geq 0.99 \  Q_{n,m} (0)$.
\end{Theorem}

It should be true that $Q_{n,m}(r) > Q_{n,m}
(0)$ for $r>0$, but we have not been able to prove it
(the problem is to prove the inequality for $r$ close to 0).
But for any given $0<\mu <1$ and a given pair $(n,m)$ one could
prove that $Q_{n,m}(r) > \mu \  Q_{n,m} (0)$.

\section{  ${\bf{H}}^n$-Yamabe constants}

In this section we will prove Theorem 1.2 and Proposition 1.3.

Recall that

$$\inf_{f\in L_1^2 (\bf{H}^n ) -\{ 0 \} } \frac{\|\nabla f\|^2_2}{\|f\|^2_2} =
\frac{(n-1)^2}{4} .$$

Let

$$Y_{n,m}^s (f)= \frac{\int_{{\bf{H}}^n}a_{n+m}|\nabla f|^2+
(s -n(n-1)) f^2dv_{g_h^n}}{\|f\|_{p_{n+m}}^2} ,$$

\noindent
so that, if $(M,g)$ is a closed Riemannian manifold of volume $V$ and
constant scalar curvature $s$ then

$$Y_{{\bf{H}}^n}
({\bf{H}}^n\times M^m,g_h^n+g) =  V^{\frac{2}{m+n}}
\ \inf_{f\in L_1^2 ({\bf{H}}^n ) -\{ 0 \} }
Y_{n,m}^s (f) .$$

We can rewrite the expression of $Y_{n,m}^s (f)$ as:

$$Y_{n,m}^s (f)= \frac{a_{m+n}}{\|f\|_{p_{n+m}}^2}
\int_{{\bf{H}}^n}|\nabla f|^2+
\left( \frac{s -c_{m,n}}{a_{m+n}}-\frac{(n-1)^2}{4} \right) f^2dv_{g_h^n} .$$

It follows that if $s -c_{m,n} <0$ then there exists $f\in C_0^{\infty}
({\bf{H}}^n )$
such that $Y_{n,m}^s (f) <0 $. For each integer $k$ we can
consider  $f_k \in C_0^{\infty}
({\bf{H}}^n )$ which consists of $k$ disjoint copies of $f$.
Then $Y_{n,m}^s (f_k ) = k^{1-(2/p)} Y_{n,m}^s (f)$ and
so

$$\lim_{k\rightarrow \infty} Y_{n,m}^s (f_k ) =-\infty ,$$

\noindent
proving the last statement of  Theorem 1.2.

If  $s -c_{m,n} \geq 0$ then  $Y_{n,m}^s (f) >0 $ for all $f\in L_1^2
({\bf{H}}^n )$. To prove that the constant is strictly positive
it is enough to consider the case when $s=c_{m,n}$. But

$$ \inf_{f\in L_1^2 ({\bf{H}}^n ) - \{ 0 \} }
Y_{n,m}^{c_{m,n}} (f) =a_{m+n} S_{n,p_{m+n} -1},$$

\noindent
where $ S_{n,p_{m+n} -1}$ is the best constant in the
Poincar\'{e}-Sobolev
inequality proved in \cite[(1.2)]{Mancini}. If the infimum were
achieved in this case then the minimizing function would be a
positive smooth solution in $L_1^2 ({\bf{H}}^n )$  of
$\Delta f + (n-1)^2 /4 +f^{p-1} =0$; but such a solution does not exist by
\cite[Theorem 1.1]{Mancini}. In case $s>c_{m,n}$ then bounds
on $Y_{n,m}^s (f)$ give bounds on the $L_1^2$-norm of $f$, so
minimizing sequences are bounded in $L_1^2 ( {\bf{H}}^n )$. Then
by the usual techniques one can show convergence to a smooth positive
function in  $L_1^2 ( {\bf{H}}^n )$. This is explicitly done in
\cite[Theorem 5.1]{Mancini}.

This concludes the proof of Theorem 1.2.

\vspace{.5cm}

Proof of Proposition 1.3:

Note that for $r>0$, $Q_{n,m} (r)
= Y_{{\bf{H}}^n}({\bf{H}}^n\times S^m,(1/r) g_h^n +g_0^m )$ and
continuity at $0$ means that

$$\lim_{T\rightarrow \infty}  Y_{{\bf{H}}^n}({\bf{H}}^n\times S^m,T
g_h^n +g_0^m ) =  Y_{\re^n}(\re^n\times S^m,g^n_e+g_0^m) .$$

\noindent
If we had a closed Riemannian manifold instead of hyperbolic space,
then we would be in the situation of \cite[Theorem 1.1]{Akutagawa}.
As in \cite{Akutagawa} one has to prove

$$\limsup_{T\rightarrow \infty}  Y_{{\bf{H}}^n}({\bf{H}}^n\times S^m,T
g_h^n +g_0^m ) \leq   Y_{\re^n}(\re^n\times S^m,g^n_e+g_0^m) ,$$

\noindent
and

$$\liminf_{T\rightarrow \infty}  Y_{{\bf{H}}^n}({\bf{H}}^n\times S^m,T
g_h^n +g_0^m ) \geq  Y_{\re^n}(\re^n\times S^m,g^n_e+g_0^m) .$$

\noindent
The proof of the first inequality given in \cite{Akutagawa} does not
use compactness and works in our situation. The second inequality is
actually very simple in our case. It follows for instance from the
Proposition 4.2 in Section 4 of this article.

Now consider  $r\in (0,1]$.

$$Q_{n,m} (r) = r^{\frac{m}{n+m}}Vol_{g_0^m}^{\frac{2}{n+m}}
\inf_{f\in L_1^2  ({\bf{H}}^n ) - \{ 0 \} } \frac{a_{n+m}\|\nabla
  f\|_2^2+ s_{g_r} \|f\|_2^2}{\|f\|_{p_{n+m}}^2} $$

Let

$$F(r) = \inf_{f\in L_1^2  ({\bf{H}}^n ) - \{ 0 \} } \frac{a_{n+m}\|\nabla
  f\|_2^2+ s_{g_r} \|f\|_2^2}{\|f\|_{p_{n+m}}^2} .$$

It is clear that $F(r)$ is uniformly bounded in any
interval $[r_0 ,1]$, for $r_0 >0$. Let $f_r$ be a minimizer
for $F(r)$ (i.e. $f_r$ is a minimizer for $Q_{n,m} (r)$). We
can normalize it to have $\|f_r \|_{p_{n+m}} =1$.
Then

$$F(r) =
a_{n+m} \|\nabla
  f_r \|_2^2+ s_{g_r} \| f_r \|_2^2 \geq
\left( \frac{a_{n+m} (n-1)^2}{4} + s_{g_r} \right)  \|f_r \|_2^2 $$

Since

$$\frac{a_{n+m} (n-1)^2}{4} + s_{g_r} \geq \frac{a_{n+m} (n-1)^2}{4}
+m(m-1)
-n(n-1) $$

$$= m(m-1) - c_{m,n} >0,$$

\noindent
it follows that  $\|f_r \|_2^2$ is uniformly bounded. Then
$F$ is clearly continuous at any
$r>0$ and so $Q_{n,m}$ is continuous.

\section{Computing $Y_{{\bf{H}}^n}({\bf{H}}^n\times
S^m,g(r))$ for $r\in (0,1]$}

Let $f_r$ be a function that achieves
$ Q_{n,m}(r) = Y_{{\bf{H}}^n}({\bf{H}}^n\times
S^m,g(r))$, where we call $g(r)=  g^n_h+rg^m_0$
and $r\in (0,1]$.
Then (after normalizing it appropriately)
$f_r(x)=\varphi_r (\| x \| )$ where $\varphi_r$
is a solution of
$EQ_{\lambda ,n,q}$ with $\lambda = \lambda (r) =
\frac{-n(n-1)+r^{-1}m(m-1)}{a_{n+m}}$
and $q=p_{n+m}-1$.

Then

$$ Q_{n,m} (r) =  Y_{g(r)}(f_r)=r^{\frac{m}{n+m}}Vol_{g_0^m}^{\frac{2}{n+m}}
\frac{a_{n+m}\|\nabla
  f_r\|_2^2+s_r\|f_r\|_2^2}{\|f_r\|_{p_{n+m}}^2} $$

$$=a_{n+m} r^{\frac{m}{n+m}}V_{g_0^m}^{\frac{2}{n+m}}\|f_r\|_{p_{n+m}}^{\frac{4}{n+m-2}}$$

\noindent
(where all the norms are taken considering $f_r$ as a function on $
{\bf{H}}^n$).

In this section $r$ (and $\lambda$) will be fixed and we want to show
that we can effectively numerically compute   $Q_{n,m} (r)$, which
means that we can compute numerically $\|f_r\|_{p_{n+m}}$.

Let  $\varphi$ be the  solution of $EQ_{\lambda ,n,q}$
with $\varphi(0) = \alpha
>0$ and $\varphi'(0)=0$. Of course $\varphi$ depends only on $\alpha$
and we will use the notation $\varphi = \varphi_{\alpha}$ when we
want to make explicit this dependence. We will use the notation
$f_{\alpha} (x) = \varphi_{\alpha} (\| x \| )$.

We are interested in the cases
$\lambda \in [a_{m+n}^{-1} (m(m-1)-n(n-1)), \infty )$.
The cases when $\lambda >0$ have some qualitative differences to the
cases when  $\lambda \leq 0$.

Consider the energy function associated with  $\varphi$:

$$E=E(\varphi ):=(1/2)  (\varphi ')^2 -\lambda \varphi^2 /2  + \varphi^{q+1} /(q+1) .$$

Then

$$E' (t) = -(n-1) \frac{e^{2t}+1 }{e^{2t} -1}  (\varphi '(t) )^2 \leq 0 .$$

If a solution $\varphi$ intersects the $t$ axis,
let $b(\varphi )$ be the first point such that $\varphi (b(\varphi )) =0$.
If $\varphi$ does not cross the $t$
axis,  we define $b(\varphi)=\infty$. Note that
in the first case $\varphi ' (b_{\varphi}) <0$
and therefore $E(b(\varphi))>0$. We are  going to consider the function
$\varphi$  defined in $[0,  b(\varphi ) ]$.

We divide the  solutions $\varphi$ into these families:

\begin{itemize}
\item $N = \{\mbox{Solutions for which}\ b(\varphi ) < \infty \}$.

\item $P =\{ \mbox{Solutions which stay positive but are not in}\ L^{q+1}\}$.

\item $G =\{ \mbox{Solutions for which}\ b(\varphi ) = \infty\ \mbox{and are in}\ L^{q+1}\}$.
\end{itemize}

The minimizing solution belongs to $G$.
It is proved in \cite[Theorem 1.2]{Mancini} that there exists
exactly one such solution.
If the initial value of this solution is $\varphi(0)=\alpha_{
\lambda}$, then they also show \cite[Corollary 4.6]{Mancini} that
if $\alpha < \alpha_{\lambda}$ then $\varphi_{\alpha} \in P$ and
if $\alpha > \alpha_{\lambda}$ then $\varphi_{\alpha} \in N$.
Note that we are using the notation $f_r = f_{\alpha_{\lambda (r)}}$.

\vspace{.5cm}

To see that one can compute $Q_{n,m} (r)$ numerically we will argue
that we can numerically approximate the value of
$\alpha_{\lambda (r)}$ and that for any given $\epsilon >0$
we can explicitly find $t>0$ such that

$$ \| {f_r}_{|_{ \{ \| x \| \} >t}} \|_p < \varepsilon .$$

\vspace{.5cm}

If for some $\alpha$ the solution
$ \varphi_{\alpha}$ hits 0, then $\alpha > \alpha_{\lambda}$. In case
$ \varphi_{\alpha}$ stays positive up to some large $T$ we consider
the following:

\begin{Lemma} Fix $\lambda$, let $\alpha > \alpha_{\lambda}$
and let $f_{\alpha}$, $f_{{\alpha}_{\lambda}}$ be the corresponding
functions in hyperbolic space. Then ${\| f_{\alpha}  \|}_p \geq
  {\| f_{\alpha_{\lambda}}  \|}_p  $. Moreover, if $\alpha_i$,
$i=1,2$ are such that $\infty > b_{\alpha_1 } > b_{\alpha_2 }$
then   ${\| f_{\alpha_2}  \|}_p \geq
  {\| f_{\alpha_{1}}  \|}_p  $.

\end{Lemma}

\begin{proof} Restrict the Yamabe functional $Y_{g(r)}$
to smooth functions with
support in closed Riemannian ball $B(0,T)$.
It can be seen that the infimum of the
functional is achieved by a smooth solution of $EQ_{\lambda ,n,q}$
which is positive in $[0,T)$ and vanishes at $T$. But there is
exactly one such solution by \cite[Proposition 4.4]{Mancini}.  It
follows that if $b_{\alpha} =T$ then  $f_{\alpha}$ is the minimizer.
The lemma follows since the infimum is
$r^{\frac{m}{n+m}}V_{g_0^m}^{\frac{2}{n+m}}\|f_{\alpha} \|_p^{\frac{4}{n+m-2}}$.

\end{proof}

Then for some given value of $\alpha$ one can numerically compute
the corresponding solution $\varphi_{\alpha}$ and decide if
$\alpha > \alpha_{\lambda}$ (in case it hits 0 at some point)
or $\alpha < \alpha_{\lambda}$ (in case its $L^p$ norm becomes bigger
than the $L^p$-norm of a solution in $N$).

In the case when $\lambda >0$ one can do it a little easier since
solutions of the equation which are in $P$ will have positive local
minimums.

\vspace{.3cm}

Finally, one can see that for a given $\epsilon >0$ one can
find $t$ such that
$ \| {f_r}_{|_{ \{ \| x \| \} >t}} \|_p < \varepsilon .$

Note first that there are  known explicit
positive lower bounds for ${\| f_r \|}_p$: this is
of course equivalent to  have lower bounds for
$Q_{n,m} (r)$ and
in \cite[Theorem 4.1, Corollary 4.2]{Ammann2}
the authors give   lower bounds for  $Y({\bf{H}}^n\times
S^m,g(r))$ (and of course
 $Y({\bf{H}}^n\times
S^m,g(r)) \leq Q_{n,m} (r)$).

Now

$$\frac{{\| f_r \|}_p^2  Q_{n,m} (r)}{r^{\frac{m}{m+n}} V_{g_0^m}^{2/(m+n)}}=
\int_{{\bf{H}}^n}a_{n+m}|\nabla f_r |^2+
\left( \frac{m(m-1)}{r} -n(n-1)\right) f_r^2dv_{g_h^n} $$

$$
\geq \left( a_{m+n} \frac{(n-1)^2}{4} + m(m-1)-n(n-1) \right)
{\| f_r \|}_2^2
.$$

Let

$$D_{m,n} = \frac{m+n-1}{m+n-2} (n-1)^2  + m(m-1)-n(n-1)>0 .$$

If $\varphi_r (t) < \varepsilon$ then
$f^p_r (x) < \varepsilon^{p-2} f^2_r (x)$ for all $x$ such that
 $\| x \| >t$. Therefore

$$\int_{\{ \| x \| > t \} } f^p_r \leq \varepsilon^{p-2} {\|
  f_r\|}_2^2 <\varepsilon^{p-2}  \frac{{\| f_r \|}_p^2
  Q_{n,m}(r)}{r^{\frac{m}{m+n}
 V_{g_0^m}^{2/(m+n)}} D_{m,n}} \leq K(\varepsilon ),$$

\noindent
where $K(\varepsilon )$ is some explicit function of $\varepsilon$
that goes to 0 with $\varepsilon$.

Upper bounds for  ${\| f_r \| }_p$ are easy to obtain
(for instance using Lemma 3.1) and this implies that given any
positive $\varepsilon$, since $\varphi_r$ is decreasing, one can
explicitly find $t$ such that $\varphi_r (t) <\varepsilon$.

Then for any given $\epsilon >0$ one can explicitly find $t$
such that the $L^p$-norm of the restriction of $f_r$ to
$\{ \| x \| >t \}$ is less than $\epsilon$.

This should make it clear  that  ${\| f_r \| }_p$
can be effectively computed numerically.

\vspace{.5cm}

To finish our description
we show  examples in each case $\lambda \leq 0$
and $\lambda >0$.

\subsection{ODE for $\lambda \leq 0$}
In this case if $t_0$ is a local minimum of a solution
$\varphi$ then $\varphi (t_0 ) <0$
and in case $t_0$ is a local maximum then  $\varphi (t_0 ) >0$.

If for some initial value the solution hits 0 we know that it belongs
to $N$. Solutions in $P$ are always decreasing and to decide if
a solution belongs to $P$ one has to apply Lemma 3.1.

The following graphic shows the  solutions of the equation
$EQ_{\lambda}$ with parameters $\lambda=-3/32$ and $q=7/3$
(which correspond to $m+n=5$ and ${\bf s}= -1/2$)
with initial condition $\varphi(0)=0.5$, $\varphi(0)=0.9$,
$\varphi(0)=1.2$, $\varphi(0)=1.9$ and $\varphi(0)=3$ respectively.
\begin{center}
 \includegraphics[scale=0.5]{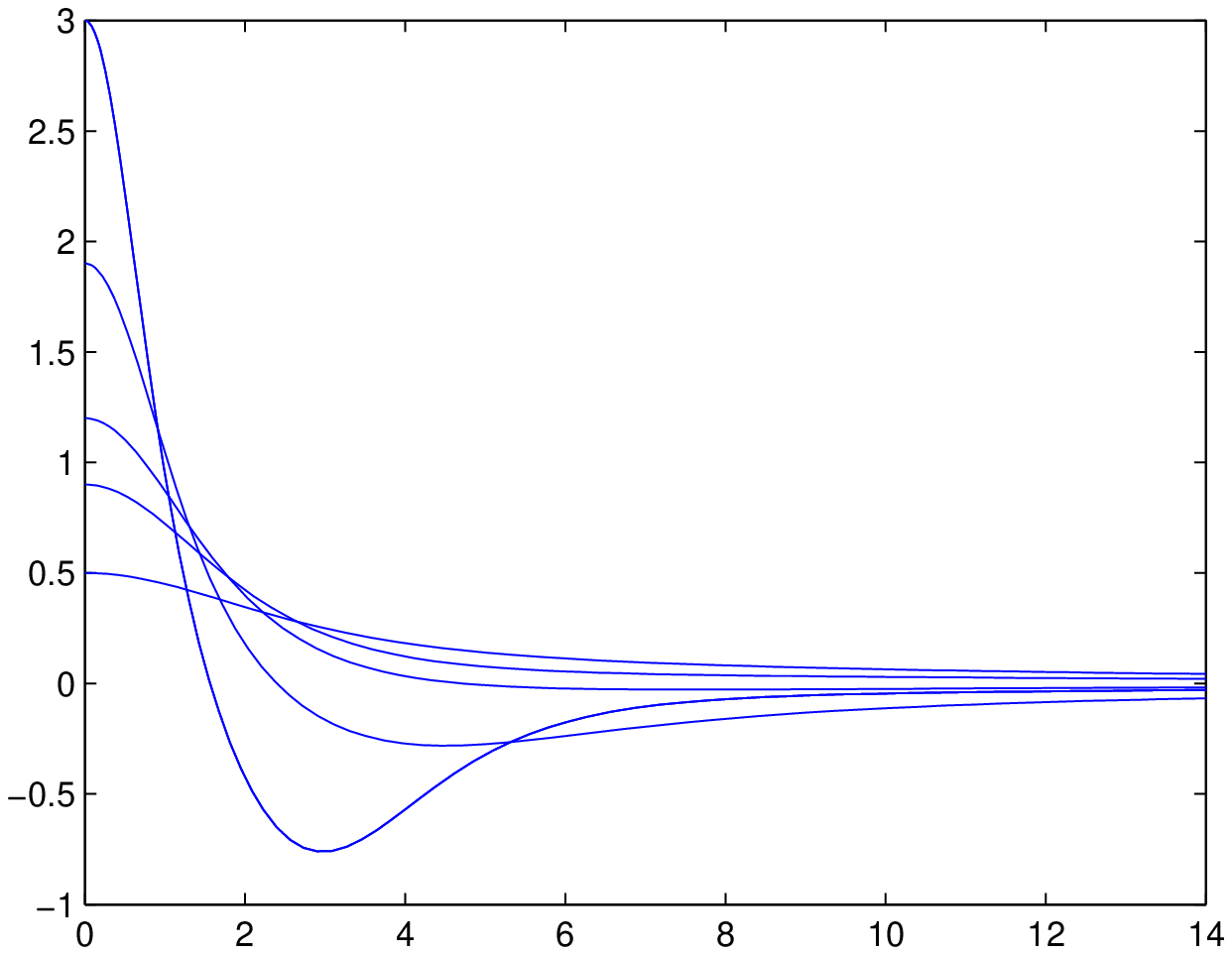}\\
\end{center}

\subsection{ODE for $\lambda  > 0$}
It is equivalent to solve

$$EQ_{\lambda}  : \ \  \  \  \  \varphi '' +(n-1) \frac{e^{2t}+1 }{e^{2t} -1}  \varphi ' = \lambda (\varphi - \varphi^q ).$$

We normalize it in this way so we always have the constant solutions 0 and 1.

 Note that if $\varphi \in N$ then $E(b(\varphi))>0$.

Since $E$ is a decreasing function, the solutions $\varphi$ are bounded.

Suppose that $t_0$ is a critical point of $\varphi$.
Then $\varphi (t_0 )<1$ if $t_0$ is a local
minimum and $\varphi (t_0) >1$ if it is a local maximum
(we are only considering $\varphi$ defined where it stays positive).
If $t_0$ is a local minimum of $\varphi$
then $E(t_0 )<0$ and $\varphi \in P \cup G$.

Now suppose that $\varphi $ is always positive and 0
is a limit point of $\varphi$. Then it follows
that $E(\infty ) = \lim_{t\rightarrow \infty} E(t)=0$.
Therefore $\varphi$ cannot have any local minimum
and $\varphi$ must be monotone decreasing (to 0).
So if $\varphi$ has a local minimum them $\varphi \in P$

The graphic below shows the solutions of the equation
$EQ_{\lambda}$ with parameters $\lambda=15/8$, $q=7/3$
(which correspond to $m+n=5$ and ${\bf s}=10$)
and with initial condition $\varphi(0)=0.3$, $\varphi(0)=2.5$
and  $\varphi(0)=2.8$ respectively.

\begin{center}

  \includegraphics[scale=0.5]{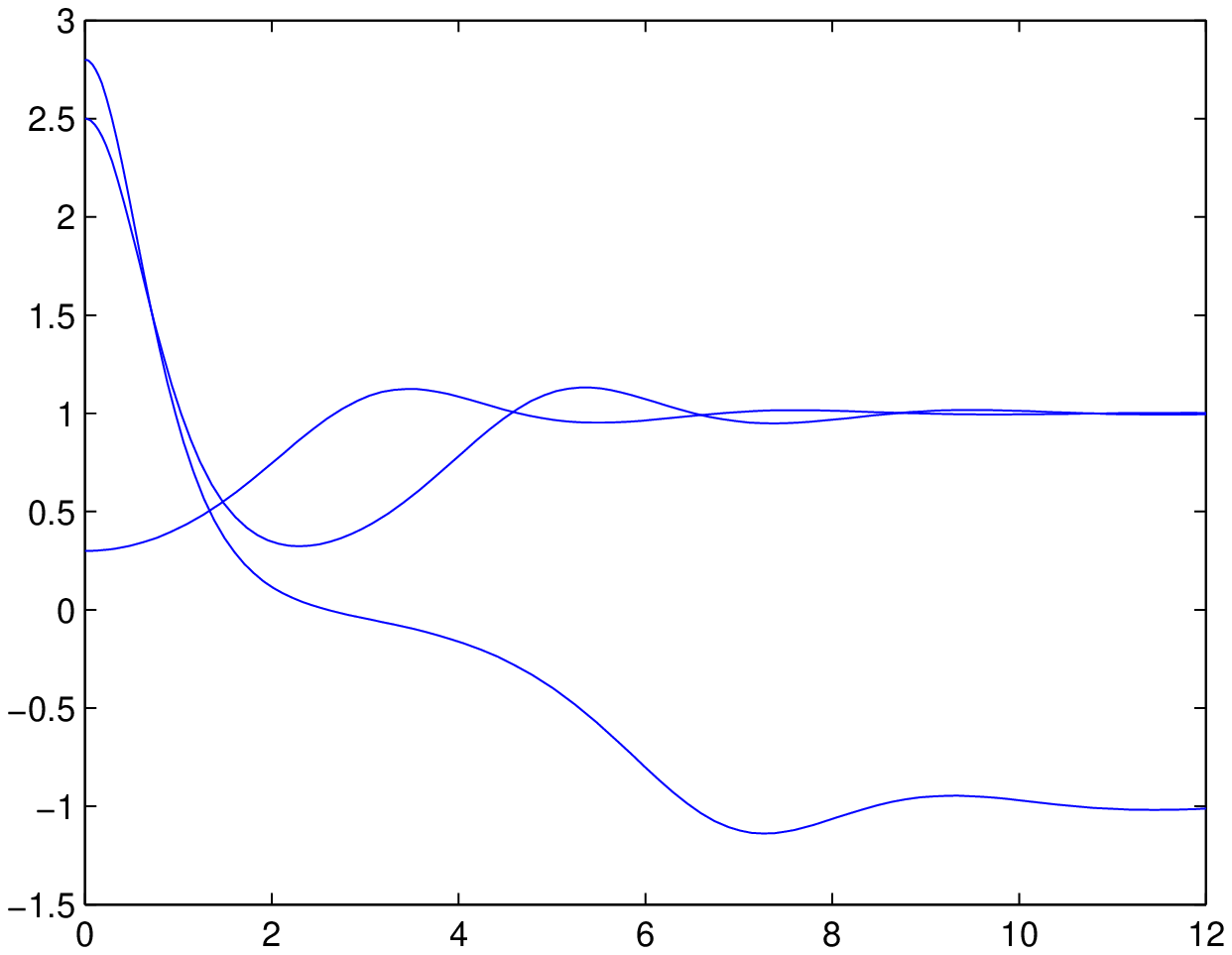}\\

\end{center}

\section{Numerical computations: proof of Theorem 1.4}

We want to estimate $Q_{n,m} (r)$ for $r\in [0,1]$.
Recall that we denote by $g(r)$ the metric $g_h^n+rg^m_0$.
Note that the product $({\bf{H}}^n\times S^m,g(r))$ is conformal
(by a constant, $1/r$) to $({\bf{H}}^n_r \times S^m,g^n_{h r}+g^m_0)$
where $g^n_{h r}$ is the hyperbolic metric of constant curvature
$-r$.  Therefore
$Y_{{\bf{H}}^n}({\bf{H}}^n\times S^m,g(r))=
Y_{{\bf{H}}^n_r}({\bf{H}}^n_r\times S^m,g^n_{h r}+g^m_0)$.
We proved that $Q_{n,m} (r)$ is continuous.

Recall also that $Q_{n,m} (1) =
Y_{{\bf{H}}^n}({\bf{H}}^n\times S^m,g(1)) =
Y({\bf{H}}^n\times S^m,g(1)) =Y(S^{n+m})$,
as was noted in \cite[Proposition 3.1]{Ammann}.
 $Q_{n,m} (0) = Y_{\re^n}(\re^n\times S^m,g^n_e+g_0^m)$
was computed in \cite{Akutagawa} and $$Q_{n,m} (0)
< Q_{n,m} (1).$$

To prove Theorem 1.4 we will use two simple results:

The following observation is a simpler case of \cite[Lemma 3.7]{Ammann}.

\begin{Lemma}\label{Y_H}
 Let $0<r_0\leq r_1$, then
 $$Y_{{\bf{H}}^n}({\bf{H}}^n\times S^m,g(r_1))\leq
(\frac{r_1}{r_0})^{\frac{m}{n+m}}Y_{{\bf{H}}^n}({\bf{H}}^n\times S^m,g(r_0)).$$
\end{Lemma}

\begin{proof} We have  that  $s_{g(r_1 )}\leq s_{g(r_0 )}$,
$dv_{g(r_0)}=(\frac{r_0}{r_1})^{\frac{m}{2}}dv_{g(r_1)}$
and  $\| \nabla f\|^2_{g(r_1)}=\|\nabla f\|^2_{g(r_0)}$
for any $f\in L^2_1({\bf{H}^n})$.
Then $Y_{g^n_h+r_1g_0^m}(f)\leq (\frac{r_1}{r_0})^{\frac{m}{n+m}}Y_{g^n_h+r_0g_0^m}(f)$ for any
$f\in L^2_1({\bf{H}^n})$  and the Lemma follows.
\end{proof}

The other simple result we will use is the following proposition. It
is proved in a more general situation in \cite[Corollary
3.3]{Ammann2}: we give a short proof of this simpler case.

\begin{Proposition}
For any small $r>0$, $Q_{n,m}(r) \geq \frac{m(m-1)-rn(n-1)}{m(m-1)}
Q_{n,m}(0) $.
\end{Proposition}

\begin{proof}For $r>0$, $Q_{n,m}(r)= Y_{{\bf H}^n} ({\bf H}^n \times S^m ,
  (1/r) g_h^n + g_0^m )$. Given any non-negative function $f\in C_0^{\infty}
({\bf H}^n )$, considered as a function in $({\bf H}^n ,(1/r) g_h^n )$,
we consider its Euclidean radial symmetrizations: this is  the radial,
non-increasing,
non-negative function
$f_* \in C_0^{\infty} (\re^n )$ such that for each $t>0$
$Vol (\{ f>t \} ) = Vol( \{ f_* >t \} )$. It is elementary
that for any $q>0$, ${\| f_* \|}_q =   {\| f \|}_q$. On the
other hand since the isoperimetric profile of
$({\bf H}^n ,(1/r) g_h^n )$ is greater than that of Euclidean space
it follows from the coarea formula that
${\| \nabla f \| }_2 \geq {\| \nabla f_* \|}_2 $.

Then, if we let ${\bf s}= -rn(n-1) +m(m-1)$, we have

$$Y_{(1/r) g_h^n + g_0^n } (f) =  V_{g_0^m}^{\frac{2}{m+n}}
\frac{\int_{{\bf{H}}^n}a_{n+m}|\nabla f|^2+ {\bf s} f^2
  dv_{(1/r) g_h^n}}{\|f\|_{p_{n+m}}^2} \geq $$

$$ \frac{-rn(n-1) +m(m-1)}{m(m-1)}
 V_{g_0^m}^{\frac{2}{m+n}}
\frac{\int_{{\bf{R}}^n}a_{n+m}|\nabla f_* |^2+ m(m-1)
 f_*^2dv_{g_e^n}}{\|f_* \|_{p_{n+m}}^2} $$

$$= \frac{-rn(n-1) +m(m-1)}{m(m-1)} Y_{g_e^m  +  g_0^n} (f_* ). $$

And the proposition follows.

\end{proof}

\subsection{${\bf H}^2 \times S^2$}

 It follows from Proposition 4.2 that
if $r\in [0, 0.01]$ then $Q_{2,2}(r) \geq 0.99 Q_{2,2}(0) =58.81076$.
It is  known that
$Q_{2,2} (1)= Y({\bf H}^2 \times S^2 , g_h^2+g_0^2 )=Y(S^{4})
=61.56239 > Q_{2,2} (0) =
Y_{\re^2}(S^2 \times \re^2 )= 59.40481 $.

Let $$s_2=\Big(\frac{0.99Q_{2,2}(0)}{Q_{2,2}(1)}\Big)^{2}.$$

Then $0.99Q_{2,2}(0)=s_2^{1/2}Q_{2,2}(1)$. By Lemma 4.1 it follows
that  $Q_{2,2}(s)\geq 0.99Q_{2,2}(0)$  for any $s\in [s_2,1]$. On the
other hand, as explained in Section 3, we can numerically compute
$Q_{2,2} (s_2) =   61.55039>0.99Q_{2,2}(0)=58.81076$. Let
$$s_3=\Big(\frac{0.99Q_{2,2}(0)}{Q_{2,2}(s_2)}\Big)^{2} \ s_2=0.83317.$$
Since $0.99Q_{22}(0)=(s_3/s_2)^{1/2}Q_{22}(s_2)$,  by Lemma 4.1 and the inequality above   $Q_{2,2}(s)\geq 0.99 Q_{2,2}(0)$ if $s\in [s_3,1]$.
 Following this procedure we found a finite succession $s_i$ with
 $i=1,\dots,126$ such that
 $$s_{i+1}=\Big(\frac{0.99Q_{2,2}(0)}{Q_{2,2}(s_i)}\Big)^{2} \ s_i,$$ $Q_{2,2}(s_i)>0.99Q_{2,2}(0)$ and $s_{126}< 0.01$. Then by Lemma 4.1 and Proposition 4.2 $Q_{2,2}(r)\geq 0.99Q_{2,2}(0)$ for all $r\in [0,1]$.

In the following table we exhibit some values of the succession $s_i$:

\small

\begin{center}
\begin{tabular}{||c|c|c|c|c|c|c||}\hline
$i$&21&42&63&84&105&126\\ \hline
$s_i$&0.22732 &0.09051 & 0.04630 &0.02641 & 0.01593& 0.00992\\ \hline
$Q_{2,2}(s_i)$& 60.42277&59.87433 &59.65783 & 59.55268& 59.49515& 59.46143 \\\hline
\end{tabular}
\end{center}
\normalsize
\vspace{0.2 cm}

The graph of $Q_{2,2}$ in $[0,1]$ is :
\begin{center}
\includegraphics[scale=0.5]{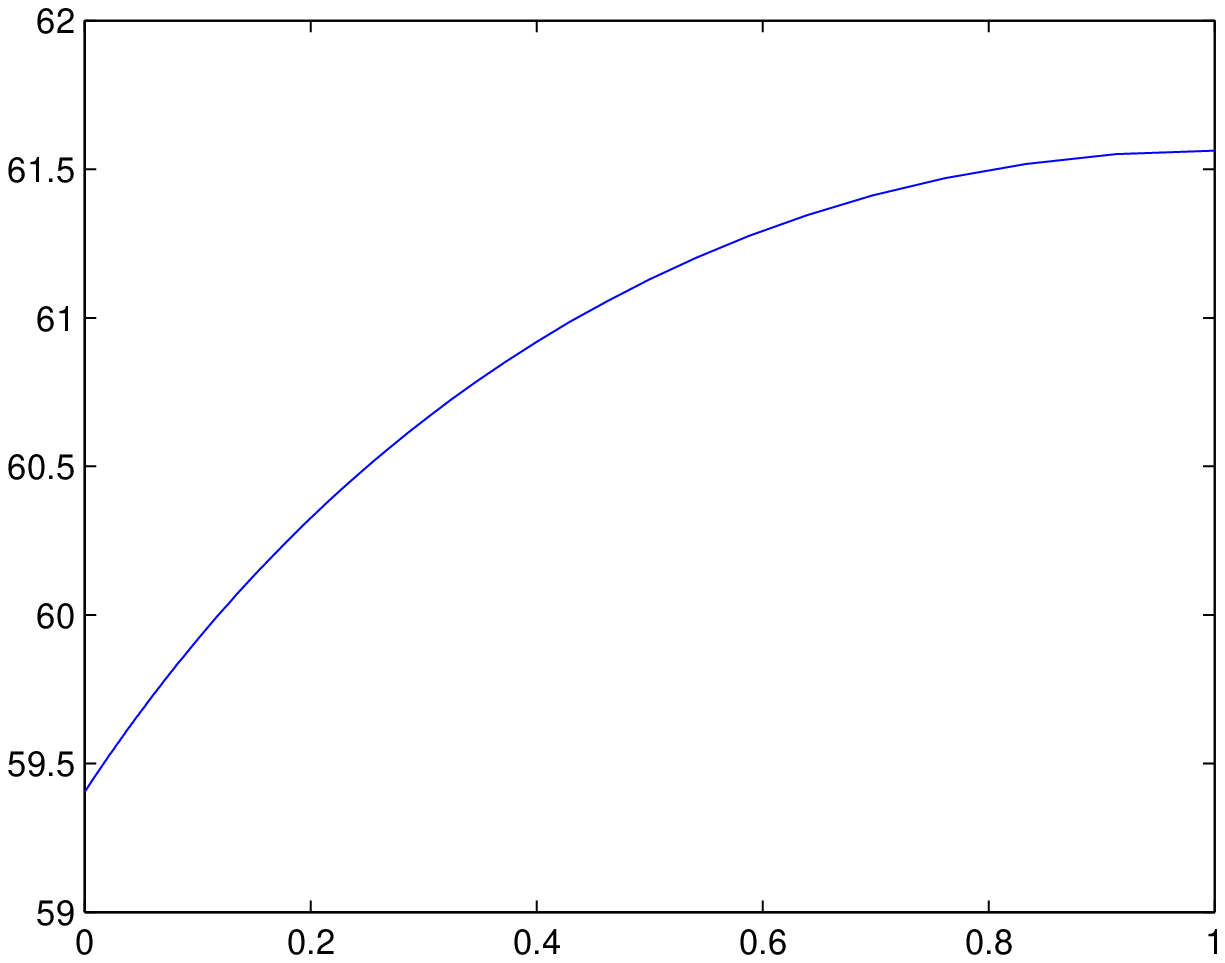}\\
\end{center}

\subsection{${\bf H}^2 \times S^3$} By Proposition 4.2
$Q_{2,3}(r)\geq 0.99Q_{2,3}(0)$ for $r\in[0, 0.03]$. It is  known that
$Q_{2,3}(0)=78.18644$ and $Q_{2,3}(1)=78.99686$. As in the case
${\bf{H^2}}\times S^2$ we can numerically compute a finite succession
$s_i$ with $i=1,\dots ,152$, such that $s_1=1$,
$$s_{i+1}=\Big(\frac{0.99Q_{2,3}(0)}{Q_{2,3}(s_i)}\Big)^{\frac{5}{3}}
\ s_i,$$ $Q_{2,3}(s_i)>0.99Q_{2,3}(0)$ and $s_{152}<0.03$. Therefore, Theorem 1.4  for  $(n,m)=(2,3)$ follows from Lemma 4.1  and Proposition 4.2.
The following table includes some values of the succession $s_i$:
\small

\begin{center}
\begin{tabular}{||c|c|c|c|c|c|c||}\hline
$i$&25 &51 &76 &101 &126 &152\\ \hline
$s_i$&0.46075 & 0.22854&0.12886 & 0.07706  & 0.04774 & 0.02968 \\ \hline
$Q_{2,3}(s_i)$&78.79217 & 78.55030 &78.40924 &78.32559 &78.27483 & 78.24226\\\hline
\end{tabular}
\end{center}

\vspace{0.2 cm}

The graph of $Q_{2,3}$ in $[0,1]$ is :
\begin{center}
\includegraphics[scale=0.5]{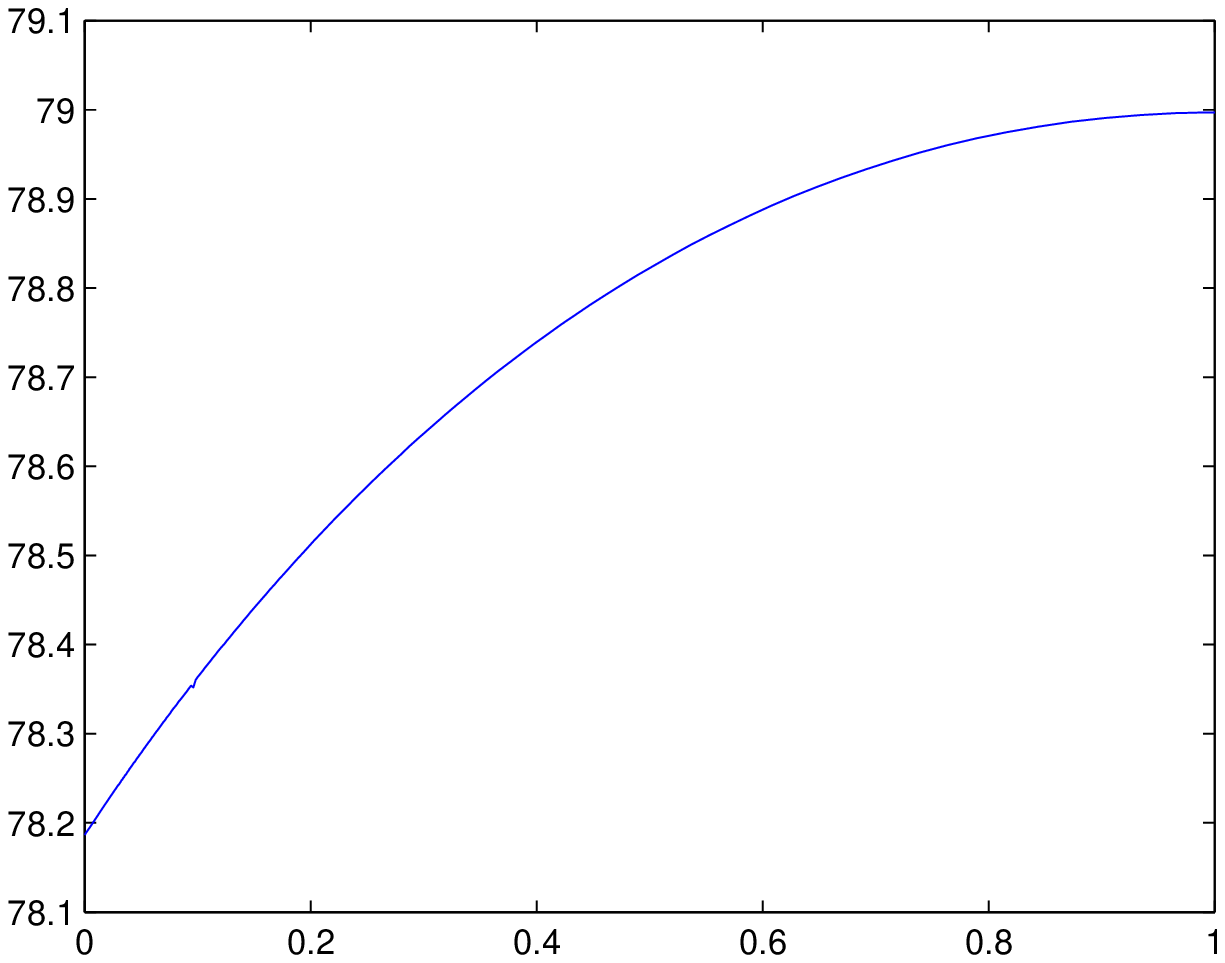}\\
\end{center}

\subsection{${\bf H}^3 \times S^2$} Recall that
$s_r\leq 0$ for $r\geq 1/3$. As in  the cases ${\bf{H^2}}\times S^2$
and ${\bf{H^2}}\times S^3$  we  found a succession $s_i$ with the
properties  described above, which proves the Theorem 1.3 in this
situation.
In this case $Q_{3,2}(1)=78.99686$, $Q_{3,2}(0)=75.39687$
and the last term $s_{132} < 1/300$.

\begin{center}
\begin{tabular}{||c|c|c|c|c|c|c||}\hline
$i$ &9 &33 & 57&81 &105 & 132\\ \hline
$s_i$ & 0.36158& 0.07155&0.02794 & 0.01315  &0.00668  & 0.00325 \\ \hline
$Q_{3,2}(s_i)$&77.77070 &76.03779  &  75.66151&75.52397& 75.46201 &75.42872 \\\hline
\end{tabular}
\end{center}

\vspace{0.2 cm}
The graph of $Q_{3,2}$ in $[0,1]$ is :

\begin{center}
\includegraphics[scale=0.5]{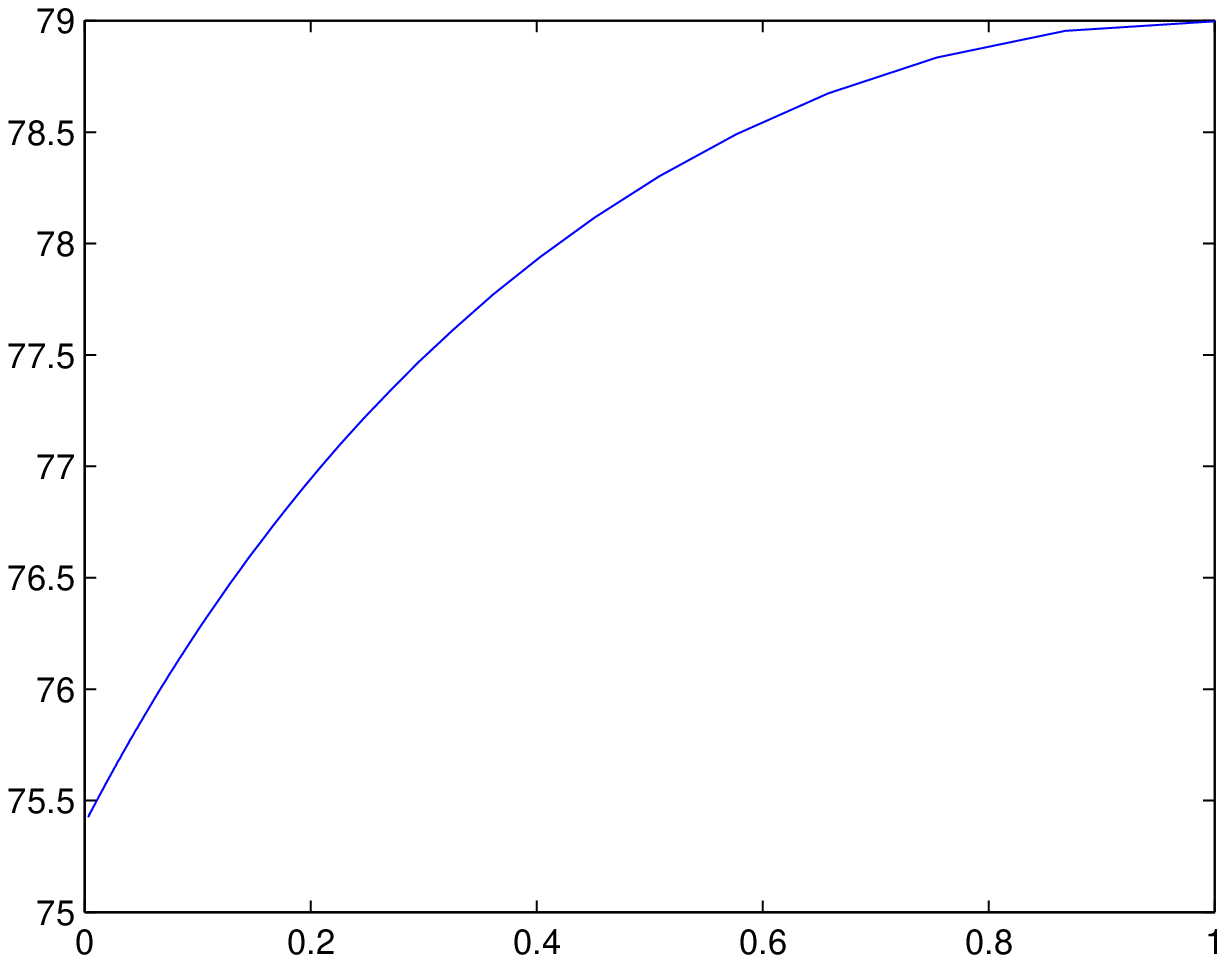}\\
\end{center}


\begin{thebibliography}{aa}

\bibitem{Botvinnik} K. Akutagawa, B. Botvinnik,
{\it Yamabe metrics on cylindrical manifolds},
Geom. Funct. Anal. {\bf 13} (2003), 259-333.

\bibitem{Akutagawa} K. Akutagawa, L. Florit, J. Petean, {\it  On
Yamabe constants of Riemannian products}, Comm. Anal. Geom.
{\bf 15} (2007), 947-969.

%\bibitem{Almeida} L. Almeida, L. Damascelli, Y. Ge,
%{\it A few symmetry results for
%nonlinear elliptic PDE on non compact manifolds},
%Annales de l'Institut Henri Poincare (C)
%Nonlinear Analysis {\bf 19} (2002), 313-342.


\bibitem{Ammann} B. Ammann, M. Dahl. E. Humbert,
{\it Smooth Yamabe invariant and surgery}, ArXiv:0804.1418,
to appear in J. Differential Geometry.

\bibitem{Ammann2} B. Ammann, M. Dahl, E. Humbert, {\it Low dimensional
surgery and the Yamabe invariant}, ArXiv: 1204.1197.

\bibitem{Aubin} T. Aubin, {\it Some Nonlinear Problems in
Riemannian Geometry}, Springer Monographs in Mathematics,
Springer, 1998.

%\bibitem{Hebey} E. Hebey, {\it Nonlinear analysis on manifolds: Sobolev spaces and inequalities}, Courant
% Lectures Notes in mathematics.



\bibitem{Kim} S. Kim, {\it Scalar curvature on noncompact complete
Riemannian manifolds}, Nonlinear Anal. {\bf 26} (1996), 1985-1993.



\bibitem{Mancini} G. Mancini, K. Sandeep, {\it On a semlinear elliptic equation in ${\bf H}^n$},
Annali della Scuola Normale Superiori di Pisa {\bf 7} (2008), 635-671.



\bibitem{Ruiz}  J. M. Ruiz, { \it Results on the existence of
the Yamabe minimizer of $M^m\times \re^n$}, J. Geom. Phys.
{\bf 62} (2012), 11-20.




%\bibitem{Trudinger} N. Trudinger, {\it Remarks concerning the conformal
%deformation of Riemannian structures on compact manifolds}, Ann. Scuola
%Norm. Sup. Pisa {\bf 22} (1968), 265-274.


%\bibitem{Yamabe} H. Yamabe, {\it On a deformation of Riemannian structures
%on compact manifolds}, Osaka Math. J. {\bf 12} (1960), 21-37.


\end{thebibliography}
\end{document}